\documentclass[12pt]{amsart}
\usepackage{graphicx}
\usepackage{amsmath}
\usepackage{amstext}
\usepackage{amsfonts}
\usepackage{amsthm}
\usepackage{amssymb}

\newtheorem{theorem}{Theorem}

\newtheorem{lemma}{Lemma}

\newcommand{\Ar}{\mathbb{R}}

\newcommand{\ep}{\epsilon}

\begin{document}

\title{On Regularity for $J$-Holomorphic Maps}
\author{Max Lipyanskiy}
\date{}
\address{Simons Center for Geometry and Physics }
\email{mlipyan@gmail.com}
\maketitle

\begin{abstract}  
We provide a short proof that an $L^2_1$ and $J$-holomorphic curve is in fact smooth.  As an application, we deduce a removal of singularity theorem for curves of finite energy.   
\end{abstract}
\section{Statement of Results}
 Given a manifolds $\Sigma$. Let $L^p_k(\Sigma)$ denote the Sobolev space of functions with all derivatives up to order $k$ in $L^p$. Such functions may be valued in some finite dimensional vector space.  Let $M$ be a compact smooth manifold with a smooth almost complex structure $J$.  For convenience, we will fix an embedding $i:M\subset \Ar^N$ and assume that $J$ extends to smooth almost complex structure in a neighborhood of $M$.  Let $\Sigma$ be a Riemann surface with complex structure $j$.  We consider a map (defined almost everywhere) $u:\Sigma\rightarrow M$ to be $L^2_1$ if the corresponding map $i\circ u:  \Sigma \rightarrow M$ is in $L^2_1$.   Furthermore, such a map is said to be $J$-holomorphic if $$du\circ j=J\circ du$$ almost everwhere.  
\begin{theorem}
Suppose $u:\Sigma \rightarrow M$ is $L^2_1$ and $J$-holomorphic. We have $u \in C^\infty(\Ar^N)$.
\label{thm}
\end{theorem}
Let $D$ be the unit open disk in the plane and let $D^*$ be the punctured disk.  As a corollary, we deduce:
\begin{theorem}
\label{cor}
Suppose $u:D^*\rightarrow M$ is $J$-holomorphic and has $$E(u)=\frac{1}{2}\int_{D^*}|du|^2<\infty$$  We have that $u$ extends smoothly to a $J$-holomorphic map on $D$.  
\end{theorem}
There is a vast literature on $L^2_1$-regularity for harmonic maps going back to Morrey. See \cite{Lin} for references.  \\\\
\textbf{Acknowledgement.}  We wish to thank Tom Mrowka for suggesting the key lemma in this paper. We also would like to thank the Simons Center For Geometry and Physics for their hospitality while this work was being completed.   
\section{Proofs}
Our proof of regularity will be based on the following application of Stokes' theorem.
\begin{lemma}
Let  $f,g \in L^2_1(\Sigma)$ be functions with compact support.  We have $d f \wedge d g \in L^2_{-1}(\Sigma)$.  
\end{lemma}
\begin{proof}
Take test function $h\in L^2_1(\Sigma)$.  Note that in since $df$, $dg \in L^2(\Sigma)$, a priori $df \wedge dg \in L^2(\Sigma)$ while $h \not\in L^\infty(\Sigma)$.  Therefore, it is not clear how to define $\int_\Sigma h\wedge df \wedge dg$.  \\
  Take a $3$-manifold $Y$ with $\partial Y=\Sigma$.  There exists $\tilde{f},\tilde{g},\tilde{h} \in L^2_{3/2}(Y)$ that extend the given $f,g,h$ from $\Sigma$ to $Y$.  By solving the Dirichlet problem,  such extensions can be taken to depend continuously on the given $f,g,h$.  Since $L^2_{1/2}(Y)\rightarrow L^3(Y) $ in dimension 3, we have $d\tilde{h} \wedge  d \tilde{f} \wedge d \tilde{g} \in L^1(Y)$.  Finally, Stokes theorem implies that $$\int_\Sigma hdf \wedge dg=\int_Y d\tilde{h} \wedge  d \tilde{f} \wedge d \tilde{g}$$  for smooth $h$ so we may set the RHS as the definition of the integral on the LHS.  
\end{proof}
\noindent \textbf{Remark.}  As a slight generalization, note that one may take the functions to be matrix valued.  In addition, the $L^2_{-1}$-norm of $df \wedge dg$ depends only on the $L^2$-norms of $df$,  $dg$ and not on $f$, $g$.  Indeed, taking replacing $f$ with $f+const$ we may assume $||f||_{L^2_1}\leq C ||df||$.  \\\\
We now turn to the proof of the theorem.  As the theorem is local in $\Sigma$ we will focus on the case of $D_r$ - a disk of radius $r$ in the plane.
\begin{lemma}
Given $u:D\rightarrow  M$ as above.  If $u\in L^2_{1,loc}(D)$, then $u$ is smooth.
\end{lemma}
\begin{proof}
Let $(s,t)$ be coordinates on $D$.  We will establish that $u\in L^2_{3/2}$. Since $L^2_{3/2}\rightarrow L^p_1$ for $p>2$, higher regularity is standard (see \cite{MS}) and follows from Sobolev multiplication theorems.\\  Let  $\Delta=\partial_s^2+\partial_t^2$.  Since $\partial_su+J(u)\partial_t u  =0$, we apply $\partial_s-J(u)\partial_t$ to deduce that $$\Delta u +\partial_sJ \partial_tu-J\partial_t J \partial_t u=0$$
This equation holds in the weak sense.   Now, using the fact that  $J\partial_t J=-\partial_t J J$ and $J\partial_tu=-\partial_s u$ we deduce that 
$$\Delta u +*dJ\wedge du=0$$
where $*$ is the Hodge star operator on $\Sigma$.  \\
\indent We  rewrite the equation as $\Delta u+T(u)=0$. With $T(u)=*(dJ\wedge du)$. $T$ defines a continuous operator $L^{2}_{3/2}\rightarrow L^2_{-1/2}$ in view of the embedding $$L^2_{1/2}\cdot L^{2}_{1/2}\cdot L^2 \rightarrow L^1$$  In addition, $T$ defines a continuous operator $L^{2}_{1}\rightarrow L^2_{-1}$.  This follows from the previous lemma.  By rescaling the disk, we may assume that the $L^2$-norm of $dJ$ is as small as we like.  Therefore, the norm of $T$ is as small as we like on the relevant Sobolev spaces.  \\
\indent Take a bump function $\phi$ with support on $D_{1/2}$.  $\phi u$ satisfies $\Delta (\phi u)+T(\phi u) =g\in L^2$. Since $T$ is small, the solution to the Dirichlet problem implies there exists a unique $v\in L^2_{3/2}(D)$ such that  $\Delta (v)+T( v) =g$ and $v=0$ on $\partial D$.  Viewed as an element of $L^2_{1}$, $v$ satisfies the same equation and thus $v=\phi u$.  Therefore, $u\in L^2_{3/2}$ as desired.  
\end{proof}
We now prove the removable singularities theorem.  This follows from theorem $\ref{thm}$ together with the following standard lemma:
\begin{lemma}
Given a bounded smooth $u:D^*\rightarrow \Ar^n$ with $E(u) < \infty$, we have that $u \in L^2_1$ as a map from $D$.  
\end{lemma}
\begin{proof}
Since $u\in L^\infty$, we need to check that $du$ is the weak derivative of $u$ on $D$.   Take a bump function $\phi:D\rightarrow \Ar$ with $\phi=1$ outside the $1/2$-ball   and equal to 0 near the origin. Let $\phi_\ep(s,t)=\phi(s/\ep,t/\ep)$.  Set $u_\ep=\phi_\ep u$ and note that $u_\ep$ is smooth.  Take a smooth test function $f$.  We have 
$$\langle u, df\rangle_D =\lim_{\ep \rightarrow 0}\langle u_\ep, df \rangle_D =\lim_{\ep \rightarrow 0}\langle du_\ep, f \rangle_D=\lim_{\ep \rightarrow 0}\langle (du)\phi_\ep+ud\phi_\ep, f \rangle_D $$
Since $\lim_{\ep \rightarrow 0}\langle (du)\phi_\ep, f \rangle_D =\langle (du), f \rangle_D$, we need only show that $$\lim_{\ep \rightarrow 0}\langle ud\phi_\ep, f \rangle_D=0$$ So see this, note that $d\phi_\ep$ has support on $D_\ep$ and is bounded by $C\ep^{-1}$ for some uniform $C>0$.  Since $u$ and $f$ are bounded, $\langle u d\phi_\ep, f \rangle_D\leq C' \pi \ep^2 \ep^{-1}$.
  \end{proof}

\end{document}